\documentclass[11pt]{article}

\usepackage{graphicx}
\usepackage{amssymb}
\usepackage{amsmath}
\usepackage{caption}
\usepackage{subcaption}
\usepackage{subfloat}
\usepackage{bbold}
\usepackage{comment}
\usepackage[export]{adjustbox}
\usepackage{amsfonts}

\newtheorem{theorem}{\bf Theorem}[section]
\newtheorem{proposition}{\bf Proposition}[section]
\newtheorem{lemma}{\bf Lemma}[section]
\newtheorem{corollary}{\bf Corollary}[section]
\newtheorem{remark}{\bf Remark}[section]
\newtheorem{definition}{\bf Definition}[section]

\newtheorem{example}{\bf Example }[section]

\usepackage{color}

\newcommand{\blue}{\textcolor[rgb]{0.00,0.00,1.00}}
\newenvironment{proof}{
\begin{trivlist}
\item[\hspace{\labelsep}{\bf\noindent Proof. }] }{\hfill $\Box$\end{trivlist}
\par}

\date{\empty}

\title{
\huge\bf Stochastic comparisons, differential entropy and varentropy for  distributions induced  by probability density functions\thanks{
Accepted for publication in {\em Metrika}.}\\ 
}

\author{
\bf Antonio Di Crescenzo$^{(1)}$,\; Luca Paolillo$^{(2)}$,\;  Alfonso Su\'arez-Llorens$^{(3)}$
\\
\\
\normalsize (1)  Dipartimento di Matematica, Universit\`a degli Studi di Salerno\\ 
\normalsize Via Giovanni Paolo II, 132, 84084 Fisciano (SA), Italy \\
\normalsize Email: adicrescenzo@unisa.it \ -- \ ORCID: 0000-0003-4751-7341 \\
\\
\normalsize (2)  Dipartimento di Matematica, Universit\`a degli Studi di Salerno\\ 
\normalsize Via Giovanni Paolo II, 132, 84084 Fisciano (SA), Italy \\
\normalsize Email: lpaolillo@unisa.it \ -- \ ORCID: 0000-0001-7146-4863 \\
\\
\normalsize (3)  Dpto.\ Estad\'istica e Investigaci\'on Operativa, Universidad de C\'adiz\\
\normalsize Facultad de Ciencias, Campus Universitario \\
\normalsize R\'io San Pedro s/n, 11510 Puerto Real, C\'adiz, Spain \\
\normalsize Email: alfonso.suarez@uca.es \ -- \ ORCID: 0000-0002-7679-9328
}

\begin{document}
 
\maketitle

\begin{abstract}
Stimulated by the need of describing useful notions related to information measures, 
we introduce the `pdf-related distributions'. These are defined in terms of  
transformation of absolutely continuous random variables 
through their own probability density functions. 
We investigate their main characteristics, with reference to 
the general form of the distribution, the quantiles, and 
some related notions of reliability theory. 
This allows us to obtain a characterization 
of the pdf-related distribution being uniform for distributions of exponential and Laplace type as well. 
We also face the problem of stochastic comparing the pdf-related 
distributions by resorting to suitable stochastic orders.
Finally, the given results are used to analyse properties and to compare 
some useful information measures, 
such as the differential entropy and the varentropy. 

\medskip\noindent
{\bf Keywords}: Differential Entropy; Differential Varentropy; Stochastic orders; Information Measures; Decreasing Rearrangement. 

\medskip\noindent
{\bf Mathematics Subject Classification (MSC)}: 60E05; 60E15; 62E99; 94A17

\end{abstract}


\section{Introduction}\label{sec:intro}
It is well known that various information measures for absolutely continuous 
random variables are given by expectations of functions of probability densities.  
Typical examples in this area involve the differential entropy and the related varentropy. 
These notions have been largely used in several contexts in order to describe the 
information content and the variability of stochastic systems.  However, in some cases 
the results dealing with these measures are not easily manageable, and require non trivial 
efforts in order to appropriately compare the quantities under consideration. 
Stimulated by the need of dealing with effective probabilistic and statistical 
tools for assessing the above mentioned information measures, 
in this paper we introduce the so-called `pdf-related distributions'. These are 
constructed by means of transformation of absolutely continuous random variables 
through their own probability density functions (pdf's). 
\par
Let us recall that, given an absolutely continuous random variable $X$ having pdf $f(x)$, 
the interest on transformations of the form $Y=f(X)$ stems in various statistical contexts. 
\\
(a) \ The expectation $\mathbb{E}[f(X)]$ is involved in the computation 
of the relative asymptotic efficiency of the Wilcoxon test relative to the $t$-test 
(cf.\ Section 1 of Hodges and Lehmann \cite{Hodges}, for instance). Additionally, it is also used in Ahmad and Kochar \cite{ahmko} for testing the classical dispersive ordering between two univariate random variables.
\\
(b) This quantity is also employed in Information Theory for constructing suitable measures of uncertainty. 
Indeed, the term $\int f(x)^2\, {\rm d}x$ is involved in the definition of the `R\'enyi entropy' of order 2 (see, for instance, 
Nadarajah and Zografos \cite{NadarajahZografos}) and of the `differential entropy' of $X$ as the expectation of the transformation $-\log(f(X))$ (cf.\ Definition D.2 of 
Lad {\em et al.} \ \cite{Lad2015}). 
\\
(c) The analysis of the distribution function of $f(X)$, for $f$ continuous, arises in problems concerning 
the so-called `confidence level' of $f$, defined as $\mathbb{P}[f(X)\leq f(a)]$ and used in 
Rommel {\em et al.}\ \cite{Romm2021} for applications in Functional Data. 
\par
Clearly inspired by the previous studied, in this paper 
we investigate the main properties of the pdf-related distributions, with special reference to 
(i) the general form of their distribution functions, (ii) the connections with 
various notions of interest in reliability theory (such as the residual lifetime),  
(iii) the determination of quantiles. 
We analyse pdf-related distributions generated by general distributions. 
However, special attention is devoted to the case when the underlying distributions possess 
monotone or unimodal pdf's. Here, unimodality refers to strictly monotone pdf's 
or to pdf's that are first strictly increasing and then strictly decreasing. 
\par
The special form that is involved in their definition allows to tackle the 
problem of comparing the pdf-related distributions by resorting to suitable stochastic 
orders, such as the usual stochastic order, the dispersive order, the convex trasform order, 
the star order and the kurtosis order. Hence, in this framework the main results involve 
both location and variability orderings, and are obtained by techniques based on variability, convexity and rearrangement of pdf's. 
\par
As main application,   the pdf-related distributions are finalized to construct a 
stochastic framework aimed to compare the basic information measures for absolutely continuous 
random variables.  This allows to come to new results on the differential entropy, on the   varentropy, and on the varentropy of residual lifetimes. 
\par
This is the plan of the paper. In Section \ref{sect:Background} we recall some useful notions on random lifetimes and 
residual lifetimes, also with reference to the related information content, the differential entropy and the differential 
varentropy. Then, necessary results on unimodal distributions and their residual lifetime distributions 
are recalled. In Section \ref{sec:pdfreldist} we introduce the notion of distributions induced by probability density functions, namely pdf-related distributions. 
Under suitable monotonicity or unimodality assumptions we investigate their distributions and the related residual lifetime distributions, expressed in terms of quantiles. 
We also obtain a characterization  of  the pdf-related distribution being $(0,1)$-uniform  for distributions of exponential and Laplace type. 
Section \ref{sec:stocorders} focus on results based on stochastic orders, 
mainly finalized to perform stochastic comparisons between pdf-related distributions, also with reference to rearrangements 
of distributions. In Section \ref{secstvarent} further stochastic comparisons are performed for the 
pdf-related distributions, and are also finalized to compare the differential entropy and varentropy of random lifetimes and of residual lifetimes.  The new results of the paper are essentially located in Sections 4 and 5. 
\par
Throughout the paper we denote by $[X|B]$ a random variable having the same distribution of $X$ conditional on $B$. 
Moreover, we write $Y=_{st}X$ when $X$ and $Y$ are identically distributed. Furthermore, $\phi'(x)$ denotes the 
derivative of any differentiable function $\phi(x)$. 
%
\section{Background}\label{sect:Background}
To improve the readability of the paper, we establish the regularity conditions such as the random variable $X$ is absolutely continuous having a strictly increasing cumulative distribution function (cdf) $F(x)=\mathbb{P}(X \leq x)$. We will denote by $\overline{F}(x)=1-F(x)$ the survival function and by $f(x)$ the pdf supported on a real interval $S_X \subseteq \mathbb{R}$. Finally, given $u\in (0,1)$, we will denote by $F^{\, -1}(u)$ the quantile function of $X$.

Now we provide some key definitions that we will use along the paper. Let $X$ be a lifetime random variable such that $F(0)=0$. Given a unit which has survived up to time $t$, its residual life is given by
\begin{equation}\label{eq:Xt}
X_t:=[X-t|X>t], \qquad t\in S_X. 
\end{equation}
By denoting $F_t(x)=\mathbb{P}(X_t\leq x)$ the cdf of $X_t$, the survival function and pdf of $X_t$ are given respectively by: 
\begin{equation} 
\label{eq:survfunreslifetime}
\overline{F}_t(x) =1-F_t(x) =  \frac{\overline{F}(x+t)}{\overline{F}(t)}, 
\end{equation}
\begin{equation} 
\label{eq:pdfreslifetime}
f_t(x)= \frac {{\rm d}{F_t(x)}}{{\rm d}{x}}  
=  \frac{f(x+t)}{\overline{F}(t)}.
\end{equation} 

The conventional approach used to characterize the failure distribution of $X$ 
is either by its (instantaneous) hazard rate function
\begin{equation}
\lambda(x)=\frac{f(x)}{\overline{F}(x)}
=\lim_{h\to 0^+} \frac{1}{h} \mathbb P[X\leq x+h|X>x], 
\qquad x\in S_X,
\label{eq:defhazfunct}
\end{equation} 
or by the cumulative hazard rate function of $X$, 
\begin{equation}\label{eq:Lambda}
\Lambda(x)=-\log{\overline{F}(x)}
= \int_0^x \lambda(s)\, {\rm d}s,  \qquad x\in S_X,
\end{equation}
both functions playing a relevant role in numerous contexts. On the other hand, $X$ is said to be Increasing Failure Rate (IFR) if $\lambda(x)$ increases in $x$. In the same manner we
define DFR (Decreasing Failure Rate) if  $\lambda(x)$ decreases in $x$. 
\par
Another key concept is given by the {\em information content} of $X$ defined as the following random variable 
\begin{equation}\label{eq:defICX}
 IC(X):=-\log f(X).
\end{equation}
This  random variable  plays a relevant role in Information Theory, since it describes the amount of uncertainty 
contained in the outcome of $X$. Indeed, it is well-known that the {\em differential entropy} of $X$ is given by the 
expectation of $IC(X)$, i.e.\ (cf.\ Cover and Thomas \cite{CoverThomas1991}) 
\begin{equation}\label{eq:defHX}
 H(X):= \mathbb{E}[IC(X)]= - \mathbb{E}[\log f(X)]
 = -\int_{-\infty}^\infty{f(x)\log{f(x)}}\,{\rm d}x.
\end{equation}
In this framework, a relevant notion that measures the variability of the information content is the  
{\em differential varentropy} of $X$, given by (cf.\ Fradelizi {\em et al.}\ \cite{Fradelizi2016}) 
\begin{eqnarray}\nonumber
V(X) \!\!\!\! & := & \!\!\!\! \mathrm{Var}[IC(X)]= \mathrm{Var}[-\log f(X)]
 =\mathbb{E}[(IC(X))^2]-[H(X)]^2
 \nonumber \\
 & =& \!\!\!\! \int_{-\infty}^\infty{f(x)[\log{f(x)}]^2}\,{\rm d}x- \left[\int_{-\infty}^\infty{f(x)\log{f(x)}}\,{\rm d}x\right]^2.
 \label{eq:defvarentropy}
\end{eqnarray}
\par
The previous definitions can be also stated for the residual life. It is well known that the residual entropy is given by  (cf.\  Ebrahimi \cite{Ebr1996} and 
Ebrahimi and Kirmani \cite{EbrKirm1996}) 
\begin{equation}\label{eq:resentropy}
 H(X_t)= \mathbb{E}[IC(X_t)]
 =-\int_t^\infty{\frac{f(x)}{\overline{F}(t)}\log{\frac{f(x)}{\overline{F}(t)}} \,{\rm d}x}, 
 \qquad  t\in S_X.
\end{equation}
The following alternative expressions hold, for $t\in D$:
$$
 H(X_t) = -\Lambda(t)-\frac{1} {\overline{F}(t)} \int_t^\infty {f(x)}\log{f(x)} \,{\rm d}x
 = 1-\frac{1}{\overline{F}(t)}\int_t^\infty f(x)\log{\lambda(x)} \,{\rm d}x.
$$
The residual varentropy is  (cf.\ Di Crescenzo and Paolillo \cite{DiCrPaol2020}) 
\begin{equation} \label{eq:resvarentropy}
\begin{split}
V(X_t) &= {\rm Var}[IC(X_t)]
 = \int_t^\infty{\frac{f(x)}{\overline{F}(t)}\left(\log{\frac{f(x)}{\overline{F}(t)}}\right)^2 {\rm d}x} -[H(X_t)]^2 \\
&=\frac{1}{\overline{F}(t)}\int_t^\infty{f(x)\left[\log{f(x)}\right]^2 {\rm d}x} - \left[\Lambda(t)+H(X_t)\right]^2,
\end{split}
\end{equation}
where $\Lambda(t)$ and  $H(X_t)$ are given in (\ref{eq:Lambda}) and  (\ref{eq:resentropy}), respectively. See, also, Goodarzi {\em et al.}\ \cite{Goodarzi2016} and \cite{Goodarzi2017} 
for  various bounds to the variance of measures in dynamic settings. 
%

We finalize recalling the concept of unimodality which simplifies notably the computation of the distribution function of the transformation $f(X)$ as we will see later on. There are different ways to study the notion of unimodal density on
$\mathbb{R}$. We consider one of the most general ways based on the
following definition (see Sudhakar and Kumar (1988), page 2).
\begin{definition}
\label{unimodal}
A random variable $X$ or its distribution function $F$ is
called unimodal about a mode (or vertex) $m_F$ if $F$ is convex on
$(-\infty ,m_F)$ and concave on $(m_F,\infty )$.
\end{definition} 
A simple consequence of the Definition \ref{unimodal} is that if
$F$ is unimodal about $m_F$, then apart from a possible mass at
$m_F$, $F$ is absolutely continuous and if this is the case then
the unimodality of $F$ about $m_F$ is equivalent to the existence
of a density $f$, which is non-decreasing on $(-\infty ,m_F)$ and
non-increasing on $(m_F,\infty )$. Note that this density function
$f$ may be constant in a set of positive measure, or not be
continuous in a countable number of points on its support. The uniform, Cauchy, Weibull, Gamma, exponential and normal distributions are some examples that satisfies Definition \ref{unimodal}.

\begin{remark} \label{interv}
Let $X$ be a random variable under the regularity conditions having a unimodal cdf $F$. If $y \in(0,\infty)$ belongs to the image of the density function $f$, it is easily seen from Definition \ref{unimodal} that the set $\{x \in \mathbb{R}: f(x) > y \}$ is always an interval, where we will denote by $l_y$ and $u_y$ the lower and upper limit of that interval, respectively. It is also known that the mode, $m_F$, satisfies that $l_y\leq m_F \leq u_y$. Additionally, if $f$ is symmetric it is apparent that $l_y+u_y=2m_F$. 
\end{remark}

\begin{remark} \label{residunimod}
It is apparent from the Definition \ref{unimodal} and the expression \eqref{eq:survfunreslifetime} that the unimodality of the underlying random variable $X$ implies the unimodality of the residual life distribution $X_t$, for all $t\in S_X$. Additionally, the mode of $X_t$ is achieved at $0$ for all $t\geq m_F$.
\end{remark}
\begin{remark} \label{prop:quantreslifetime}
Let $X$ be a random variable under the regularity conditions having a cdf $F$. Given $u \in (0,1)$ and $t\in S_X$, it is straightforward to show that $F_t(x)$ also satisfies the regularity conditions and its inverse at $u$ is given by $F_t^{\,-1}(u)=F^{\,-1}(1-(1-u)\overline{F}(t))-t$.
\end{remark}

\section{Pdf-related distributions}\label{sec:pdfreldist}

Let us now define a special type of distribution, named ``pdf-related distribution''. This is defined by means of a transformation 
expressed by the pdf of a given baseline random variable.
\begin{definition}
\rm Let $X$ be a random variable under the regularity conditions having a pdf $f(x)$. Then, the random variable $f(X)$ is named as the pdf-related random variable of $X$. 
For any $y\in \mathrm{Im}^+(f)$ the distribution function of $f(X)$ is defined as
\begin{equation}
		K(y):=\mathbb{P}(f(X)\leq y).
		\label{ed:defK}
	\end{equation}
\end{definition}
\par
We first emphasize how $K(y)$ is affected by location and scale changes. Let $X$ be a random variable under the regularity conditions; a straightforward computation shows that 
\begin{equation} \label{Kaffine}
K_{aX+b}(y):= \mathbb{P}(f_{aX+b}(aX+b) \leq y)= K_{X}(|a|y), 
\end{equation}
where $a, b\in \mathbb{R}$, $a\neq 0$. In such a case it is clear that $|a|f_{aX+b}(aX+b)=_{st} f(X)$. It is also remarkable that $f_{-X}(-X)=_{st} f(X)$, therefore $X$ and $-X$ share the same pdf-related distribution.   
\par
Obtaining the distribution function $K(y)$ is a difficult exercise in practice. However, the assumption of unimodality considerably simplifies its computation. In this case, 
from Remark \ref{interv} we obtain that
\begin{equation}
   K(y): =1-\mathbb{P}(f(X) > y)=F(l_y)+\overline{F}(u_y),
\qquad y\in \mathrm{Im}^+(f).
 \label{eq:KyUnimod}
\end{equation}
\begin{example}
From  Eq.\ (\ref{eq:KyUnimod}) 
it is not hard to see that if $X$ has pdf  
$$
 f(x)=\left( 1-\left(1-\frac{1}{\alpha}\right)x\right)^{1/(\alpha-1)}, \qquad 0<x<\frac{\alpha}{\alpha-1}, 
$$
for $\alpha>1$, then $K(y)=y^\alpha$, for $0\leq y\leq 1$. 
Note that the case $\alpha \to 1$ corresponds to the exponential distribution, $X \sim {\rm Exp}(1)$. 
\end{example}
\par
It is worth pointing out that a pdf-related distribution is not necessarily continuous. 
Indeed, for instance, if $X$ is uniformly distributed on $(a,b)$, then clearly $f(X)$ is degenerate in $(b-a)^{-1}$. 
\par
Hereafter we obtain a characterization of the pdf-related distribution being $(0,1)$-uniform based on  exponential and Laplace distributions. It involves the lower and upper limits introduced in 
Remark \ref{interv}.  
\begin{proposition}
Let $X$ be a random variable under the regularity conditions and having a unimodal pdf $f(x)$ with mode $m_F\in \mathbb R$. Then, $f(X)$ is uniform on $(0,1)$ if and only if 
$$
 l_y-u_y=\ln(y), \quad \forall y  \in \mathrm{Im}^+(f).
$$
Additionally, (i) if $f$ is symmetric with support $\mathbb{R}$, then $f(x)=\exp(-2|x-m_F|)$, $x\in \mathbb{R}$,\\ 
(ii) if the support of $f$ is $[m_F,+\infty)$, then $f(x)=\exp(m_F-x)$, $x\in [m_F,+\infty)$, and finally,\\ 
(iii) if the support of $f$ is $(-\infty, m_F]$, then $f(x)=\exp(x-m_F)$, $x\in(-\infty, m_F]$.
\end{proposition}
\begin{proof} 
From Eq.\ \eqref{eq:KyUnimod} we know that $K(y) = F(l_y)+\overline{F}(u_y)$, for all $y\in \mathrm{Im}^+(f)$ and of course $f(m_F)=1$, $f(l_y)=f(u_y)=y$. Then 
\begin{eqnarray*}
f(X) \sim U(0,1) & \Longleftrightarrow &  K'(y)=1, \quad \forall y\in \mathrm{Im}^+(f) \\
 & \Longleftrightarrow & f(l_y) \frac{dl_y}{dy}-f(u_y)\frac{du_y}{dy}=1, \quad \forall y\in \mathrm{Im}^+(f) \\ 
 &  \Longleftrightarrow & 
\frac{dl_y}{dy}-\frac{du_y}{dy}=\frac{1}{y}, \quad \forall y\in \mathrm{Im}^+(f) \\ 
& \Longleftrightarrow & l_y-u_y=\ln(y),\quad \forall y  \in \mathrm{Im}^+(f).
\end{eqnarray*}
The rest of the proof follows by taking into account that, in the 3 cases, $l_y-u_y$ is given by $2(l_y-m_F)$, $m_F-u_y$ and $l_y-m_F$, respectively.
\end{proof}
\begin{example}
An example of non-symmetric pdf leading to the uniform distribution is provided by the following parametric density family 
$$
f(x) = \left\{ \begin{array}{lll} \exp\left(\displaystyle\frac{x-m_F}{1-\alpha}\right), & 
 x \leq m_F \\  
\\
\exp\left(\displaystyle\frac{-(x-m_F)}{\alpha}\right),  &  
x \geq m_F. \end{array} \right.   
$$
Indeed, in this case we have that $l_y-u_y=\ln(y)$, $\forall y  \in \mathrm{Im}^+(f)$, so that $f(X)$ is uniform on $(0,1)$, for all $\alpha \in (0,1)$. Note that $f(x)$ is symmetric 
only for $\alpha =0.5$. 
\end{example}
\par
In the following, when the pdf $f$ has support $S_X=(a,b)$, we adopt the following notation: 
$f(a)=\displaystyle\lim_{x\to a^+} f(x)$ and $f(b)=\displaystyle\lim_{x\to b^-} f(x)$. 
\par
As application of the construction of pdf-related distributions, 
hereafter we determine the cdf of the pdf-related random variable of the residual lifetime defined in (\ref{eq:Xt}), 
denoted as 
\begin{equation}\label{eq:defKty}
 K_t(y):=\mathbb{P}(f_t(X_t)\leq y)
 =\mathbb{P}(f(X)\leq y \overline{F}(t) \, | \, X>t),
\end{equation}
with the last identity due to (\ref{eq:pdfreslifetime}). 
\begin{proposition}\label{prop:distrfunrespdfreluni}
Let $X$ be an absolutely continuous random variable having support $S_X=(0,b)$.
\begin{itemize}
\item[(a)] If the pdf $f$ is continuous and strictly monotonic in $(t_0,b)$ for a given $t_0\in (0,b)$, 
then for all $t\in (t_0,b)$ one has 
\begin{equation}\label{eq:cdfKty2e3}
	K_t(y)
	=\left\{
	\begin{array}{lll}
	\displaystyle \frac{ \overline{F}(f^{\,-1}(y\overline{F}(t)))}{ \overline{F}(t)}, 
	& y\in \mathrm{Im}^+(f_t) = \left(\frac{f(b)}{\overline{F}(t)}, \lambda(t)\right), 
	& \hbox{if $f$ is decreasing,}
	\\[3mm]
	1- \displaystyle \frac{\overline F(f^{\,-1}(y\overline{F}(t)))}{ \overline{F}(t)}, 
	& y\in \mathrm{Im}^+(f_t) = \left(\lambda(t), \frac{f(b)}{\overline{F}(t)}\right),
	& \hbox{if $f$ is increasing,}
	\end{array}
	\right.
\end{equation}
where $f^{-1}$ denotes the inverse of the restriction to $(t_0,b)$ of $f$.
\item[(b)] Let the pdf $f$ be continuous in $S_X$, unimodal and symmetric, strictly increasing for $x \in (0,m]$ and strictly decreasing for $x \in [m, b)$, with $m=b/2$. Then,  for all $t\in (0,m]$ one has
\begin{equation}\label{eq:cdfKty4}
	\displaystyle K_t(y)=
	\left\{\begin{array}{ll}
	\displaystyle {\frac{  {F}( f^{\,-1}(y\overline{F}(t)))}{ \overline{F}(t)}}, \ 
		& y\in \left(\frac{f(0)}{\overline{F}(t)}, \lambda(t)\right)
		\\[+4mm]
	\displaystyle \frac{2 {F}(f^{\,-1}(y\overline{F}(t)))- F (t)}{ \overline{F}(t)}, \   
	& y\in \left[\lambda(t), \frac{f(m)}{\overline{F}(t)}\right)
		\end{array}
	\right.
	\end{equation}
where $f^{-1}$ is the inverse of the restriction to $(0,m]$ of $f$.
\end{itemize}
\end{proposition}
\begin{proof}
In the case (a), if $f$ is decreasing, from (\ref{eq:defKty}) for all $t\in (t_0,b)$ and $y\in \mathrm{Im}^+(f_t)$ 
we have $K_t(y)=\mathbb{P}(X\geq f^{\,-1}(y\overline{F}(t)) \, | \, X>t)$. 
Eq.\ (\ref{eq:cdfKty2e3}) then follows straightforwardly. 
When $f$ is increasing the procedure is analogous. 
In the case (b), when $y\in \left(\frac{f(0)}{\overline{F}(t)}, \lambda(t)\right)$ the function $K_t(y)$ can be obtained 
as in case (a), whereas when $y\in \left[\lambda(t), \frac{f(m)}{\overline{F}(t)}\right)$ we have 
$$
	K_t(y)=\mathbb{P}(t<X\leq f^{\,-1}(y\overline{F}(t)) \,| \, X>t)
	+\mathbb{P}(b-f^{-1}(y\overline{F}(t)) \leq X\leq b \,| \, X>t).
$$
The expression given in  (\ref{eq:cdfKty4}) thus follows by noting that 
$\overline{F}(b-x)=F(x)$ for all $0\leq x\leq b$. 
\end{proof}
\par
Clearly, the cdf given in (\ref{eq:cdfKty4})  is continuous in $y=\lambda(t)$, with 
$$
 K_t(\lambda(t))=\frac{F(t)}{\overline{F}(t)}, \qquad t\in (0,m],
$$  
where  the right-hand-side is the odds function of $X$ evaluated at $t\leq m$. 
\par
We conclude this section with an example concerning the distribution treated in Proposition \ref{prop:distrfunrespdfreluni}. 
\begin{example}\label{ex:hyprv}
	Let $X$ be a Pareto-type random variable having pdf and cdf given respectively by 
$$
	f(x)=\frac{1}{(1+x)^2}, \qquad x\in(0,+\infty), \qquad F(x)=\frac{x}{1+x}, \qquad x\in[0,+\infty).
$$
Since $f$ is strictly decreasing for all $x\in(0,+\infty)$, with inverse
$f^{-\,1}(y)=y^{-1/2}-1$, $ y\in(0,1)$, from (\ref{eq:cdfKty2e3}) we have that, for all $t>0$
$$
	K_t(y) 
	=\left(y(1+t)\right)^{1/2},  
	\qquad y\in\Big(0,\frac{1}{1+t}\Big).
$$
\end{example}
%

\section{Results based on stochastic orders} \label{sec:stocorders}

With the goal of comparing information measures as the ones given by the differential entropy and the varentropy, we next recall some classical results involving stochastic orders. Roughly speaking, stochastic orders deals about the comparison of two random quantities and can be divided in three kind of comparisons: magnitude, variability and shape. Let $X$ and $Y$ be two random variables having cdf's $F(x)$ and $G(x)$ and pdf's $f(x)$ and $g(x)$, respectively. The following definitions will be relevant.
\par
(i) According to the magnitude, we will say that $X$ is said to be smaller than $Y$ in the usual stochastic order, denoted by $X \leq_{st} Y$, if ${F}(x) \geq {G}(x)$, for all $x\in \mathbb R$. In some sense, the previous inequality says that $X$ is less likely than $Y$ to take on large values. The usual stochastic order deals with a characterization in terms of the expectations of increasing transformations, i.e., $X\leq_{st}Y$ if and only if 
\begin{equation}
\label{eq1.A.7}
E[g(X)] \leq E[g(Y)]
\end{equation}
holds for all increasing functions $g$ for which the expectations exist. In particular, just considering $g(x)=x$, the usual stochastic order implies the comparison of the means.
\par
(ii) According to the variability, we will say that $X$ is smaller than $Y$ in the dispersive order, denoted by $X\leq_{disp} Y$, if 
$F^{\,-1}(v)-F^{\,-1}(u) \leq G^{\,-1}(v)-G^{\,-1}(u)$, for all $0< u \leq v < 1$. The dispersive order has a  characterization in terms of the density functions evaluated at quantiles, i.e., $X\leq_{disp} Y$, if and only if
\begin{equation} \label{disp}
f(F^{\,-1}(u)) \geq g(G^{\,-1}(u)),   \qquad \textrm{ for all 
} u\in(0,1). 
\end{equation}
\par
(iii) According to the shape, we will say that $X$ is smaller than $Y$ in the convex transform order, denoted by $X\leq_c Y$, if 
$G^{-1}(F(x))$ is convex in $x$ on the support of $F(x)$. The convex transform order can be also characterized in terms of the density functions evaluated at quantiles, i.e.,  $X\leq_c Y$,  if and only if
\begin{equation} \label{convex}
\frac{f(F^{\,-1}(p))}{g(G^{\,-1}(p))}, \quad \textrm{is increasing in \ } p\in(0,1).
\end{equation}
\par
(iv) According to the shape, if $X$ and $Y$ are non-negative random variables we will say that $X$ is smaller than $Y$ in the star order, denoted by $X\leq_{*} Y$, if 
\begin{equation} \label{star}
\frac{G^{-1}(p)}{F^{-1}(p)}, \quad \textrm{is increasing in } p\in (0,1) \quad  \Longleftrightarrow \quad \log(X) \leq_{disp} \log(Y).
\end{equation}
\par
(v) According to the shape, if $X$ and $Y$ have symmetric density functions, we will say that $X$ is smaller than $Y$ in the kurtosis order, denoted by $X\leq_k Y$, if $G^{\,-1}(F(x))$ is concave for all $x<Me(X)$, or, equivalently, $G^{\,-1} (F(x))$ is convex for all $x>Me(X)$, where $Me(X)=F^{\,-1}(0.5)$ represents the median of $X$. It is pointed out in Oja \cite{Oja1983} that the kurtosis order can be characterized in terms of the convex transform order, i.e., $X\leq_k Y$ if and only if 
\begin{equation} \label{kurtosisconvtr}
	|X-Me(X)|\leq_c |Y-Me(Y)|.
\end{equation}

\par
For a comprehensive review on the background of stochastic orders we refer the reader to Shaked and Shantikumar \cite{ShakShant2007}. 
\par
To conclude the exposition of the stochastic orders we recall the notion of rearrangement of a function which has played a key role in many inequalities in literature. It is described in the seminal book of Hardy {\em et al.}\ \cite{Hardy1934} and has been studied in the context of entropy, randomness, majorization and dispersion in Hickey \cite{Hickey1984}, \cite{Hickey1986} and Fern\'andez-Ponce and Su\'arez-Llorens \cite{FPonce2003}, among other authors. We next recall the decreasing rearrangement of a density function and some of its important
properties. 
\begin{definition}\label{def:rearrdensfun}
Let $X$ be an absolutely continuous random variable having a pdf $f(x)$. The  decreasing rearrangement of $f(x)$, denoted by $f^*(x)$, is given by 
$$
f^*(x)=\sup\{c: m(c)>x \}, \qquad x>0,
$$
where $m(c)=\mu \{t: f(t)>c\}$, $\mu$ denoting the Lebesgue measure.
\end{definition}
From Hardy {\em et al.}\ \cite{Hardy1929}, we know that the decreasing rearrangement satisfies the following identity 
$$ 
\int_0^\infty f^*(x) \,\mathrm{d}x = \int_{-\infty}^\infty f(x) \,\mathrm{d}x =1.
$$
Then $f^*$ can be considered as a pdf of a particular random variable that we will denote by $X^*$. 
\par
The following implication is well-known (see Hardy {\em et al.} \cite{Hardy1929}, Theorems 9 and 10). Let $X$ and $Y$ be absolutely continuous random variables having pdf $f$ and $g$, respectively. Then 
\begin{equation} \label{major}
X^* \leq_{st} Y^* \quad \iff \quad  \int_{-\infty}^{\infty} u(f(x))\,\mathrm{d}x \leq \int_{-\infty}^{\infty} u(g(x))\,\mathrm{d}x,
\end{equation}
for all continuous and concave functions $u$. It is worth mentioning that $X^* \leq_{st} Y^*$ was equivalently described as $\int_{0}^{t} f^*(x) \,\mathrm{d}x \geq \int_{0}^{t} g^*(x) \,\mathrm{d}x$ for all $t>0$ in Hickey \cite{Hickey1984} where it was interpreted as continuous majorisation.
\par
The ordering $X^* \leq_{st} Y^*$ implies a particular comparison between $f(X)$ and $g(Y)$. Rewriting the right side of expression \eqref{major}, we obtain that $X^* \leq_{st} Y^*$ is equivalent to  
\begin{equation}\label{chacrear}
 \mathbb E\left[\frac{u(f(X))}{f(X)}\right] \leq 
 \mathbb E\left[\frac{u(g(Y))}{g(Y)}\right],
\end{equation}
for all  continuous and concave functions $u$. The latter relation  provides a class of measures of entropy. Indeed, as it is interpreted in Hickey \cite{Hickey1984}, just taking $u(x)=-x\log (x)$, one has that if $X^* \leq_{st} Y^*$ then $H(X) \leq H(Y)$ follows directly from the definition of the differential entropy given in (\ref{eq:defHX}). 
To finalize, the following remarks are straightforward and they will be useful later on. 
\begin{remark} \label{rm:frearrstoceq}
If $X$ has a symmetric and unimodal pdf $f(x)$, then $X^*=_{st} 2 |X-Me(X)|$. 
\end{remark}
\begin{remark} \label{rm:rearrstrdec}
If $X$ has a strictly decreasing pdf in the interval  $(0,+\infty)$, then $f^*(x)=f(x)$ for all $x\in (0,+\infty)$, i.e., \ $X^*=_{st} X$.
\end{remark}
\par
At this point we wonder about what kind of comparisons are we interested in. Keeping in mind that we focus on the comparison of information measures, it seems logical that we are looking for comparisons between two pdf-related random variables $f(X)$ and $g(Y)$ that lead to the comparison of the differential entropy and varentropy. For example, from the previous stochastic order definitions we have the following results.
\begin{lemma}\label{cr:stocorderentr}
Let $X$ and $Y$ be two random variables under the regularity conditions having pdf's  $f(x)$ and $g(x)$, respectively, 
and having finite differential entropy. If $f(X)\leq_{st} g(Y)$, then $$H(X)\geq H(Y).$$
\end{lemma}
\begin{proof}
From the hypothesis assumption, the expression of the differential entropy given in \eqref{eq:defHX} and considering $g(x)=\log(x)$, the proof follows easily just using \eqref{eq1.A.7}.
\end{proof}
\begin{lemma}\label{cr:starordervarentr}
Let $X$ and $Y$ be two random variables under the regularity conditios having pdf's  $f(x)$ and $g(x)$, respectively, 
and having finite differential varentropy. If  $f(X)\leq_* g(Y)$, then $$V(X)\leq V(Y).$$
\end{lemma}
\begin{proof}
From expression \eqref{star}, $f(X)\leq_* g(Y)$ is equivalent to 
$\log{f(X)} \leq_{disp} \log{g(Y)}$ which implies that $\mathrm{Var}[\log{f(X)}]\leq \mathrm{Var}[\log{g(Y)}]$
(see, for instance, Section  3.B.2 of \cite{ShakShant2007}). The 
assertion then follows from the definition of the varentropy given in (\ref{eq:defvarentropy}).
\end{proof}
\par
A natural question is whether a certain stochastic order between $X$ and $Y$ implies a stochastic comparison between $f(X)$ and $g(Y)$, which in turn leads to the comparison of the information measures. The following two examples deserve to be highlighted. 
\begin{example} \label{rearrand}
The usual stochastic ordering between the rearrangements $X^*$ and $Y^*$, $X^*\leq_{st} Y^*$, implies the comparison of $f(X)$ and $g(Y)$ provided in \eqref{chacrear}, which in turn implies  the comparison of the differential entropy, $H(X) \leq H(Y)$.
\end{example}  
\begin{example} \label{dispstarex}
From \eqref{disp}, if  $X \leq_{disp} Y$ then $ f(F^{-1}(u)) \geq g(G^{-1}(u)) $ for all $u\in (0,1)$. Just considering the inverse probability integral transformation and Theorem 1.A.1 in \cite{ShakShant2007}, if $X \leq_{disp} Y$ holds then $f(X)\geq_{st} g(Y)$ and using Lemma \ref{cr:stocorderentr} we obtain that $H(X) \leq H(Y)$. For example, it is well known that if $X$ is IFR, then $X_{t_2}\leq_{disp} X_{t_1}$ for all $t_1 \leq t_2$,  see Belzunce {\em et al.}\ \cite{beletal1996} and Pellerey and Shaked \cite{pell1996}. From the previous arguments it is apparent that if $X$ is IFR then $H(X_t)$ is decreasing when $t$ increases. This result is well known in the literature, see Ebrahimi \cite{Ebr1996}. A similar result holds for DFR distributions by exchanging all inequalities and replacing decreasing for increasing in the residual differential entropy.
\end{example}

We finalize this section with a  result concerning with the comparison of affine transformations.
\begin{proposition}\label{pr:starorderlintr} 
Let $X$ be a random variable under the regularity conditions with pdf $f(x)$, and let $Y=aX+b$, $a, b\in\mathbb{R}$, $a\neq 0$, having density $g$. Then, $f(X) =_*g(Y)$ and $f(X) \geq_{st} (\leq_{st}) g(Y)$ if $|a| \geq (\leq) 1$. Therefore $V(X) = V(Y)$ and $H(X) \leq (\geq) H(Y)$ if $|a| \geq (\leq) 1$.
\end{proposition}
\begin{proof}
By using \eqref{Kaffine}, we easily obtain that $|a| K^{-1}_{Y} (u) = K_{X}^{-1}(p)$. Then, $f(X) =_*g(Y)$ follows directly by \eqref{star}. Also using \eqref{Kaffine}, we obtain that $K_{Y} (y) = K_{X}(|a| y) \geq (\leq) K_{X}(y)$ if $|a| \geq (\leq) 1$. Therefore, $f(X) \geq_{st} (\leq_{st}) g(Y)$ if $|a| \geq (\leq) 1$ just using the definition of the stochastic order. The rest of the proof is a direct consequence of Lemma \ref{cr:stocorderentr} and Lemma \ref{cr:starordervarentr}.
\end{proof}
We would like to emphasize that the relation $f(X) =_*g(Y)$ implies that both pdf-related distribution functions have proportional quantile functions. In such a case, they belong to the same scale family of distributions (see, for instance, Section 4.1 of Di Crescenzo {\em et al.}\ \cite{DiCr2016}).   

\section{Differential entropy, varentropy and stochastic orders} \label{secstvarent}
\par
As we showed in Example \ref{dispstarex}, the dispersive order, 
$\leq_{disp}$, between the underlying random variables $X$ and $Y$ implies the usual stochastic order, $\leq_{st}$, between the pdf-related distributions, $f(X)$ and $g(Y)$, which in turn provides the comparison of the differential entropy as discussed in Lemma \ref{cr:stocorderentr}. 
Additionally, we also showed in Example \ref{rearrand} that the usual stochastic order, 
$\leq_{st}$, between the rearrangements $X^*$ and $Y^*$ of the underlying random variables $X$ and $Y$ implies a particular randomness order, see expression \eqref{chacrear}, between the pdf-related distributions which in turn also provides the comparison of the differential entropy. In this section, we are interested in the comparison of the differential varentropy, i.e., we will provide the following type of results: 
\begin{equation} \label{niceresult}
X \leq_{(1)} Y \quad \Longrightarrow \quad f(X) \leq_{*} g(Y) \quad \Longrightarrow \quad V(X) \leq V(Y)     
\end{equation}
where $\leq_{(1)}$ is any particular stochastic order and the second implication follows directly from Lemma \ref{cr:starordervarentr}. This type of results helps us to provide many possible examples to compare the differential varentropy.

As we mentioned, the rearrangement captures in some sense the degree of randomness in a density function and it has a close connection with the differential entropy. Now we wonder if the rearrangement is also related with the concept of varentropy. From now on, we will assume that the density functions have no flat zones or, equivalently, the rearrangements are strictly decreasing. Next result is of the \eqref{niceresult} type.
\begin{theorem}\label{th:pdfrelrearr}
Let $X$ and $Y$ be two random variables under the regularity conditions having pdf's $f$ and $g$, respectively, with no flat zones. Then
	$$
  X^* \leq_c Y^* \quad \iff \quad 	f(X)\leq_* g(Y).
	$$
\end{theorem}
\begin{proof}
From Eq.\ (2.1) of Hickey \cite{Hickey1984}, 
the distribution function of $X^*$ can be expressed as  
$F_{X^*}(t) =\mathbb{P}(f(X)>f^*(t))$.
The use of strict inequality in the set $\{ x\in \mathbb R: f(x) > f^*(x) \}$ follows from the assumption of $f$ having no flat zones. Therefore,  one has 
$F_{f(X)}\left(f^*(t)\right)=1-F_{X^*}(t)$, $ t\in \mathbb{R}$. 
	Because $f^*$ is strictly decreasing, we have  
\begin{equation} \label{curiosity}
		F^{\,-1}_{f(X)}(p)=f^*\left(F_{X^*}^{-1}(1-p)\right), \qquad 0<p<1.
\end{equation}

Then, recalling \eqref{star} the   relation 
$f(X)\leq_* g(Y)$ is satisfied if and only if
$$
		\frac{F^{\,-1}_{g(Y)}(p)}{F^{\,-1}_{f(X)}(p)}
		=\frac{g^*\left( G_{Y^*}^{-1}(1-p)\right)}{f^*\left(F_{X^*}^{-1}(1-p)\right)} 
	 \qquad \textrm{is increasing in } p\in(0,1),
$$
which is equivalent to $ X^* \leq_c Y^*$ by virtue of \eqref{convex}.

\end{proof}

\begin{remark}
From expression \eqref{curiosity} and taking into account the inverse probability integral transform, we obtain that $f(X)=_{st} f^{*}(X^*)$. Therefore, it is clear that random variables having the same rearrangements have also the same differential entropy and varentropy. Here we provide an illustrative example. Let $X$ and $Y$ be two random variables taking values in $[-1,1]$, with density functions $f$ and $g$ given by
$$
 f(x)=1-|x|, \quad -1\leq x\leq 1, 
 \qquad
 g(x)=|x|, \quad -1\leq x\leq 1,
$$
respectively. It is easy to compute that
$$
m_{f}(c)=m_{g}(c)=\mu \left\{ x:g(x)>c \right\}= 2(1-c)I_{[0,1]}(c).
$$
Then $f^{*}(x)= g^{*}(x)=(1- x/2)I_{[0,2)}(x)$. Therefore, we have that $f^*(X^*)=_{st}g^*(Y^*)$ and thus we  conclude that $H(X)=H(Y)$ and $V(X) = V(Y)$. It is worth mentioning that $f$ is unimodal but $g$ is not.
\end{remark}


%
In practice, Theorem \ref{th:pdfrelrearr} provides many possible comparisons as we show in the following results.

\begin{corollary} \label{cr:kurtosisconvtr}
	Let $X$ and $Y$ be two symmetric unimodal random variables under the regularity conditions. Then, 
$$
	X\leq_k Y  \quad \mbox{ if and only if } \quad f(X)\leq_{*} g(Y).
$$ 
\end{corollary}
\begin{proof}
Using jointly expression \eqref{kurtosisconvtr} and Remark \ref{rm:frearrstoceq} we obtain that $X\leq_k Y$ if and only if $X^* \leq_{c} Y^*$, due to the well-known fact that the transform convex order does not depend on scale changes. The result follows easily just applying Theorem \ref{th:pdfrelrearr}. 
\end{proof}

Corollary \ref{cr:kurtosisconvtr} provides many possible comparisons in terms of the differential varentropy. We find in Arriaza {\em et al.}\ \cite{Arriaza2019} a compilation of many unimodal symmetric distributions that are ordered in the kurtosis sense. For example, we know that the classical normal, logistic and Cauchy distributions satisfy that  normal $\leq_{k}$ logistic $\leq_{k}$ Cauchy. Therefore, just applying Corollary \ref{cr:kurtosisconvtr} and Lemma \ref{cr:starordervarentr} we obtain that the differential varentropy are ordered, i.e., 
$$V(\mbox{normal}) \leq V(\mbox{logistic}) \leq V(\mbox{Cauchy}).$$

\begin{corollary}\label{cr:convtrasfstar}
Let $X$ and $Y$ be two random variables under the regularity conditions having decreasing density functions with modes $m_X$ and $m_Y$, respectively. Then, $$X\leq_{c} Y \quad \mbox{ if and only if } \quad f(X) \leq_{*} g(Y).$$ 
\end{corollary}
\begin{proof}
Since changes on location do not affect the transform convex order and using \eqref{Kaffine} we can consider $m_x=m_y=0$  without loss of generality. Using Remark \ref{rm:rearrstrdec} we obtain that $X\leq_c Y$ if and only if $X^* \leq_{c} Y^*$. The result follows easily just applying Theorem \ref{th:pdfrelrearr}.
\end{proof}
Corollary \ref{cr:convtrasfstar} also provides many possible comparisons in terms of the differential varentropy as we will see in the following results. 
\par
\begin{theorem} \label{monotvaren}
Let $X$ be a nonnegative absolutely continuous random lifetime with support $S_X$ having a cdf $F$ and a pdf $f(x)$. Let $X_t$ be its 
residual lifetime at time $t$, as defined in (\ref{eq:Xt}). Let us suppose that there exists $t_0\in S_X$ such the pdf of $X_t$, defined in  (\ref{eq:pdfreslifetime}), is strictly decreasing  $\forall t\geq t_0$. If the ratio 
$$
\frac{f(F^{-1}(1-(1-p)u))}{f(F^{-1}(1-(1-p)v))}
$$
increases (decreases) in $p\in (0,1)$, $\forall \,0< v<u<1$, then $V(X_t)$ increases (decreases) in $t\geq t_0$. 
\end{theorem}
\begin{proof}
Given $t_1$ and $t_2$ such that $t_0<t_1<t_2$ and using first Corollary \ref{cr:convtrasfstar} and then Lemma \ref{cr:starordervarentr} we just need to check that $X_{t_1} \leq_{c} (\geq_{c}) X_{t_2}$.  From the characterization of the convex transfor order given in \eqref{convex}  and from 
the expressions of the pdf of the residual life  given in \eqref{eq:pdfreslifetime} and the inverse of the distribution function of the residual life given in Remark \ref{prop:quantreslifetime}, we obtain that
$$
 X_{t_1} \leq_{c} (\geq_{c}) X_{t_2} \; \iff \;  \frac{f(F^{-1}(1-(1-p)\overline{F}(t_1)))}{f(F^{-1}(1-(1-p)\overline{F}(t_2)))}  \quad \textrm{is increasing (decreasing) in } p\in(0,1).
$$
The result follows from the hypothesis  just considering $v=\overline{F}(t_2)$ and $u= \overline{F}(t_1)$. 
\end{proof}

\par
\begin{example} \label{exweibull}
Let $X\sim Weibull(k, \lambda)$ be a Weibull distribution with shape parameter $k$ and scale parameter $\lambda$. From the expression of the cdf of $X$ given by $F(x)=1-\exp(-(x/\lambda)^{k})$, $x\geq 0$, it is a straightforward matter to compute the ratio   
$$
\frac{f(F^{-1}(1-(1-p)u))}{f(F^{-1}(1-(1-p)v))} = \frac{u}{v} \left (\frac{\ln((1-p)u)}{\ln((1-p)v)} \right)^{\frac{k-1}{k}} 
$$
for all $0< v<u<1$. It follows easily that the above ratio is increasing (decreasing) when $k > (<) 1$ and it is constant for $k =1$. It is well-known that Weibull distributions are always unimodal having decreasing density functions after the mode. Then, the pdf's of the residual lives after the mode satisfy the conditions of Theorem \ref{monotvaren}, see Remark \ref{residunimod}. Then, $V(X_t)$ increases (decreases) $\forall t \geq$ mode, for $k > (<) 1$. On the other hand, the Weibull distribution is IFR (DFR) for $k >(<)1$. Then, just using Example \ref{dispstarex} we obtain that the differential entropy and the varentropy of the residual lives have  opposite directions of growth, namely, if $k>1$ we obtain that $V(X_t)$ increases and $H(X_t)$ decreases, and the opposite for $k<1$. The case $k =1$ is just the exponential distribution where the residual lives have both differential entropy and varentropy constant. Some examples of the varentropy of the residual lives in the Weibull distributions are shown computationally in \cite{DiCrPaol2020}. 
\end{example}
We recall that in Theorem 3.4 of \cite{DiCrPaol2020} it is proved that if a random lifetime $X$ is ILR, 
i.e.\ its pdf is log-concave, then the residual varentropy  (\ref{eq:resvarentropy}) is such that 
$V(X_t)\leq 1$ for all $t$ in the support of $X$. The condition of log-concave density means that the pdf belongs to the class of strong unimodal densities. Hereafter we show a similar result making use of 
the above findings. 
\begin{theorem} \label{ifrvarem}
Let $X$ be a nonnegative random lifetime under the regular conditions with support $S_X$, and let $X_t$ be its 
residual lifetime at time $t$ as defined in (\ref{eq:Xt}). If, for a given $t\in S_X$, $X_t$ is IFR and 
its pdf (\ref{eq:pdfreslifetime}) is strictly decreasing, then $ V(X_t)\leq 1$.
\end{theorem}
\begin{proof}
Since, for a given $t\in S_X$, the residual lifetime $X_t$ is IFR, then for Theorem 4.B.11 of \cite{ShakShant2007} 
one has that $X_t\leq_c \mathrm{Exp}(1)$. 
Moreover, by assumption $X_t$ has a strictly decreasing pdf $f_t(x)$ for all $x\geq 0$. Then, using  Corollary \ref{cr:convtrasfstar} we know that
$X_t\leq_c \mathrm{Exp(1)}$ if and only if  
$$
f_{X_t}(X_t)\leq_* g(\mathrm{Exp(1})),
$$
where $g$ is the pdf of $\mathrm{Exp(1)}$.
Hence, making use of Lemma \ref{cr:starordervarentr} we find
$$
V(X_t)\leq V(\mathrm{Exp}).
$$
Finally, the assertion 
follows recalling that $V(\mathrm{Exp})=1$. 
\end{proof}
\begin{remark}
From Remark \ref{residunimod}, we have that the conditions of Theorem \ref{ifrvarem} easily hold for all unimodal IFR distributions when $t$ is greater than the mode of $X$. For example, this is the case of the residual lives of the Weibull distribution for $k>1$, as it is described in Example \ref{exweibull}. After the mode the varentropy increases but it has an upper bound equal to $1$.
\end{remark}
%

\subsection*{Acknowledgements}
{\small
The authors acknowledge support received from the Ministerio de Econom\'ia
y Competitividad (Spain) under grant
no.\ MTM2017-89577-P, by the Ministerio de Ciencia e Innovaci\'on under grant no.\ PID2020-116216GB-I00 
and by the Consejer\'ia de Econom\'ia, Conocimiento, 
Empresas y Universidad (Junta de Andaluc\'ia, Spain) under grant FEDER-UCA18-107519, and from the Italian MIUR-PRIN 2022, 
project `Anomalous Phenomena on Regular and Irregular Domains: Approximating Complexity for the Applied Sciences', No.\ 2022XZSAFN. 
\par
A.\ Di Crescenzo and L.\ Paolillo are members of the research group GNCS of  Istituto Nazionale di Alta Matematica 
(INdAM).
\par
L.\ Paolillo expresses his warmest thanks to the Departamento de  Estad\'{\i}stica e Investigaci\'on Operativa of Universidad de C\'adiz for the hospitality during the visit carried out in 2020.
}
%

\subsection*{Conflict of interest statement }

On behalf of all authors, the corresponding author states that there is no conflict of interest.


\end{document}